\documentclass{amsart}
\usepackage{amssymb,latexsym,amsmath,amsthm,enumitem,mathrsfs,geometry}
\usepackage{fullpage}
\usepackage{fancyhdr}
\usepackage{color}
\usepackage{stmaryrd}
\usepackage{mathrsfs}
\usepackage{hyperref}
\usepackage{lineno}
\usepackage{diagbox}
\usepackage{array}
\usepackage{tikz}
\usetikzlibrary{arrows}
\usetikzlibrary{positioning}
\usepackage{float}

\newtheorem{thm}{Theorem}[section]
\newtheorem*{thm*}{Theorem}

\newtheorem{lemma}[thm]{Lemma}

\newcommand{\beq}{\begin{equation}}
\newcommand{\eeq}{\end{equation}}

\newtheorem{theorem}{Theorem}

\definecolor{pink}{rgb}{1,.2,.6}
\definecolor{orange}{rgb}{0.7,0.3,0}
\definecolor{blue}{rgb}{.2,.6,.75}
\definecolor{green}{rgb}{.4,.7,.4}
\definecolor{purple}{RGB}{127,0,255}

\begin{document}

\numberwithin{equation}{section}

\title{The error term in the Ces{\`a}ro mean of the prime pair singular series}

\author[Goldston]{D. A. Goldston}
\address{Department of Mathematics and Statistics, San Jose State University}
\email{daniel.goldston@sjsu.edu}

\author[Suriajaya]{Ade Irma Suriajaya}
\address{Faculty of Mathematics, Kyushu University}
\email{adeirmasuriajaya@math.kyushu-u.ac.jp}
\thanks{$^{*}$ The second author was supported by JSPS KAKENHI Grant Number 18K13400.} 

\date{\today}

\begin{abstract}
We show that the error term in the asymptotic formula for the Ces{\`a}ro mean of the singular series in the Goldbach and the Hardy-Littlewood prime-pair conjectures cannot be too small and oscillates. \end{abstract}

\maketitle

\section{Introduction } 

Hardy and Littlewood \cite{HardyLittlewood1919} conjectured in 1919 an asymptotic formula for the number of pairs of primes differing by $k$, see also \cite{HardyLittlewood1922}. Hardy and Littlewood's conjecture is equivalent to, for $k$ even and $k\neq 0$,

\begin{equation}\label{Psi2} \psi_2(N,k):= \sum_{\substack{n,n^\prime\le N\\ n^\prime-n=k}}\Lambda(n)\Lambda(n^\prime) \sim \mathfrak{S}(k) (N-|k|) \quad \text{as} \ \ N\to \infty, \end{equation}
where 
\begin{equation}\label{SingProd}
\mathfrak{S}(k) = \begin{cases}
{\displaystyle 2 C_{2} \prod_{\substack{ p \mid k \\ p > 2}} \!\left(\frac{p - 1}{p - 2}\right)} & \mbox{if $k$ is even, $k\neq 0$}, \\
 0 & \mbox{if $k$ is odd}\end{cases}
\end{equation}
and
\begin{equation}\label{eq1.3}
C_{2}
 = \prod_{p > 2}\! \left( 1 - \frac{1}{(p - 1)^{2}}\right)
 = 0.66016\ldots.
\end{equation}
For odd $k$ the sum in \eqref{Psi2} 
has non-zero terms only when $n$ or $n^\prime$ is a power of 2, so $\psi_2(N,k) = O((\log N)^2).$ For the Goldbach problem Hardy and Littlewood conjectured an analogous formula for the number of ways an even number $k$ can be expressed as the sum of two primes, which also includes the arithmetic function $\mathfrak{S}(k)$.

The function $\mathfrak{S}(k)$ is called the singular series, a name given by Hardy and Littlewood because it first occurred as the series 
\begin{equation} \label{eq1.4} \mathfrak{S}(k) = \sum_{q=1}^\infty \frac{\mu(q)^2}{\phi(q)^2}c_q(-k),\end{equation}
where the Ramanujan sum $c_q(n)$ is defined by
\begin{equation} \label{eq1.5} c_q(n) = \!\sum_{\substack{1\le a\le q\\ (a,q)=1}}\!\!\!\;e\Big(\frac{an}{q}\Big), \quad \ \ e(\alpha)=e^{2\pi i \alpha}. \end{equation} 
Some well-known properties of $c_q(n)$ (see, e.g., \cite[Theorem 4.1]{MontgomeryVaughan2007}) are that $c_q(-n) = c_q(n)$, $c_q(n)$
is a multiplicative function of $q$, and 
\begin{equation}\label{eq1.6} c_q(n) = \sum_{\substack{d|n\\d|q}} d\mu\!\left(\frac{q}{d}\right)\! = \frac{ \mu\big(\frac{q}{(n,q)}\big)\phi(q) }{\phi\big(\frac{q}{(n,q)}\big)}.\end{equation} 

Since the singular series is a sum of multiplicative functions in $q$, it is easy to verify that \eqref{eq1.4} is equivalent to the product in \eqref{SingProd}. 
The sum in \eqref{eq1.4} is a Ramanujan series; many arithmetic functions can be expanded into these series which have the property that the first term $q=1$ is the average or expected value of the arithmetic function. Thus we see that the $q=1$ term in \eqref{eq1.4} says that $\mathfrak{S}(k)$ has average value $1$. If we consider the first two terms we have
\[ \mathfrak{S}(k) = 1+ e\Big(\!\!\!\:-\frac{k}{2}\!\;\Big)\!\!\: +\sum_{q=3}^\infty \frac{\mu(q)^2}{\phi(q)^2}c_q(-k),\]
and therefore we obtain the refinement that on average $\mathfrak{S}(k)$ is $0$ if $k$ is odd and is $2$ if $k$ is even.

The average of $\mathfrak{S}(k)$ was studied by Hardy and Littlewood in their initial paper and has been the subject of much further work. The best result known \cite{FriedlanderGoldston1995} is 
\begin{equation} \label{SingAve} \sum_{k\le x} \mathfrak{S}(k) = x - \frac12 \log x +O((\log x)^{\frac23})).
\end{equation}
In applications to the number of primes in short intervals, the Ces{\`a}ro mean naturally occurs, and we have by \cite[Proposition 2]{FriedlanderGoldston1995} and
\cite{MontgomerySound1999}
\begin{equation}\label{SingCAve} S(x) := \sum_{k\le x}(x-k) \mathfrak{S}(k) =\frac12 x ^2 - \frac12 x \log x + \frac12( 1-\gamma -\log 2\pi) x + O( x^{\frac12 + \epsilon}). \end{equation}
Notice neither one of these results implies the other result.

Vaughan \cite{Vaughan2001} refined the error term in \eqref{SingCAve} and proved it is 
\[ \ll x^{\frac12} \exp\left( -c \frac{( \log 2x)^{\frac35}}{(\log\log 3x)^{\frac15}} \right), \]
for some positive constant $c$. On the Riemann Hypothesis he showed this error is $\ll x^{\frac5{12}+\varepsilon}$.
Vaughan actually proved these results for a more general singular series needed for primes in arithmetic progressions. 

Vaughan's results show that the size of the error term in \eqref{SingCAve} depends on a zero-free region of the Riemann zeta-function, but there is no barrier at $x^{1/2}$ if the Riemann Hypothesis is true. If one assumes the Riemann Hypothesis, one heuristically finds that each complex zero of the zeta-function makes a contribution of size $\asymp x^{1/4}$, but convergence issues make it difficult to prove an explicit formula for $S(x)$ involving zeros. This problem can be avoided with smoother weights than the Ces{\`a}ro mean, and in \cite{G-S20} we obtain an explicit formula when $m\ge 2$ for
\begin{equation} S_m(x) = \sum_{k\le x} (x-k)^m \mathfrak{S}(k), \end{equation}
the Riesz mean for the singular series.

In this note we prove unconditionally an omega result for the error term that is obtained by Landau's oscillation theorem using the first zero of the zeta-function on the critical line. Let $E(x)$ be defined by the equation
\begin{equation}\label{E(x)} S(x) =\frac12 x ^2 - \frac12 x \log x + \frac12( 1-\gamma -\log 2\pi) x + E(x). \end{equation} 

\begin{theorem} We have, as $x\to \infty$, 
\begin{equation} \label{OmegaE} E(x) = \Omega_{\pm}(x^{1/4}) . \end{equation}
\end{theorem}

It is somewhat surprising that this result was not discovered earlier; this might be due to not having an explicit formula available for the Ces{\`a}ro mean of the singular series. In most applications of Landau's theorem, such as the error in the prime number theorem and the Mertens Hypothesis, one has an explicit formula as a guide. Not having that guide makes it easier to forget that Landau's theorem depends only on the singularities of the generating function, and not on proving the validity of an explicit formula.

If we assume the imaginary parts of the zeros of the Riemann zeta-function are linearly independent over the integers, or even over a small finite set of integers, we can improve the result in Theorem 1 and prove
\[
\limsup_{x\to\infty} \frac{E(x)}{x^{\frac14}} =\infty
\qquad \text{and} \qquad
\liminf_{x\to\infty} \frac{E(x)}{x^{\frac14}}= -\infty.
\]
This is done in \cite{G-S20}.

\section{The Singular Series Generating Functions 
\texorpdfstring{$F(s)$}{F(s)} and \texorpdfstring{$G(s)$}{G(s)}}

For $k\ge 1$, define
\begin{equation} \label{g(k)} g(k) = \prod_{\substack{ p \mid k \\ p > 2}} \!\left(\frac{p - 1}{p - 2}\right), \end{equation}
where we use the convention that an empty product is equal to 1. Thus $g(1)=g(2)=1$, $g(k) \ge 1$, $g(2k)=g(k)$, and $g(k)$ is a multiplicative function with $g(p^m) = g(p)$. 

We define the singular series generating functions
\begin{equation} \label{F and G} F(s) = \sum_{k=1}^\infty \frac{\mathfrak{S}(k)}{k^s}, \qquad \text{and} \quad G(s) = \sum_{k=1}^\infty \frac{g(k)}{k^s} , \end{equation}
where as usual $s=\sigma + i t$. $G(s)$ is a little easier to work with than $F(s)$, and since
\begin{equation} \label{F-G} F(s) = \sum_{k=1}^\infty \frac{\mathfrak{S}(2k)}{(2k)^s} = \sum_{k=1}^\infty \frac{2C_2g(2k)}{(2k)^s} = \frac{2C_2}{2^s} G(s),\end{equation}
it is easy to recover $F(s)$ from $G(s)$.

We first note the useful estimate
\begin{equation} \label{gbound} \mathfrak{S}(k) \ll g(k) \ll \log\log 3k .\end{equation}
To obtain this, note
\[ g(k) = \prod_{\substack{ p \mid k \\ p > 2}} \!\left(1+\frac{1}{p - 2}\right),\]
and we should expect that $g(k)$ behaves rather similarly to $\frac{k}{\phi(k)}$ where $\phi(k)$ is the Euler phi-function, which has the product formula
\[ \phi(k) = k\prod_{ p \mid k} \!\left(1-\frac{1}{p }\right).\]
The bound \eqref{gbound} now follows from
\[ \begin{split} \frac{\phi(k)g(k)}{k} & \le \prod_{\substack{ p \mid k \\ p > 2}} \!\left(1-\frac{1}{p }\right)\left(1+\frac{1}{p-2 }\right) \\& = \prod_{\substack{ p \mid k \\ p > 2}} \!\left(1+\frac{1}{p(p-2) }\right) \\&
\le \prod_{ j=3}^\infty \!\left(1+\frac{1}{j(j-2) } \right) = 2,\end{split} \]
and the well known estimate 
\[ \frac{k}{\phi(k)} \ll \log\log3 k\]
obtained from the Mertens formula. 

From \eqref{gbound}, we see that $F(s)$ and $G(s)$ are defined and converge for $\sigma >1$. We now obtain the analytic continuation of $G(s)$ to the half-plane $\sigma >-1$.
\begin{lemma} \label{G(s)lem} We have, for $\sigma >-1$, 
\begin{equation} \label{G(s)lemma}
F(s) 
= \left(\frac{4C_2}{2^{s+1}+1} \right) \frac{ \zeta(s)\zeta(s+1)}{\zeta(2s+2)}\mathcal{G}(s), \quad \text{and} \quad G(s) 
= \left(\frac{2^{s+1}}{2^{s+1}+1} \right) \frac{ \zeta(s)\zeta(s+1)}{\zeta(2s+2)}\mathcal{G}(s),\end{equation}
where 
\begin{equation} \label{Gcal} \mathcal{G}(s) = \prod_{p>2}\left(1+\frac{2}{(p-2)(p^{s+1}+1)} \right). \end{equation}
\end{lemma} 
The formula \eqref{G(s)lemma} is due to Vaughan \cite{Vaughan2001} who also showed that the line $\sigma=-1$ is a natural boundary for $G(s)$.

\begin{proof}[Proof of \eqref{G(s)lemma}]
Since $g(k)$ is multiplicative, we have
\[ \begin{split} G(s) &= \prod_p \left(1 + \sum_{r=1}^\infty\frac{g(p^r)}{p^{rs}}\right) \\ &=\left(1 + \sum_{r=1}^\infty \frac{1}{2^{rs}}\right)\prod_{p >2} \left(1 + \left(\frac{p-1}{p-2}\right)\sum_{r=1}^\infty\frac{1}{p^{rs}}\right) \\&
=\left( 1-\frac{1}{2^s}\right)^{-1} \prod_{p >2} \left(1 + \left(\frac{p-1}{p-2}\right)\frac{1}{p^s-1}\right). \end{split}\]
This last product is $\ll \prod_p( 1+ O(p^{-\sigma}))$ and therefore converges for $\sigma >1$. We now write this last equation as 
\begin{equation} \label{g_1} G(s) = \left( 1-\frac{1}{2^s}\right)^{-1} \prod_{p >2}\left(1+ g_1(p,s)\right), \qquad \text{where} \quad g_1(p,s) = \left(\frac{p-1}{p-2}\right) \frac{1}{p^s-1}. \end{equation}
Our goal is to pull out zeta-function factors from this Euler product and obtain successive new Euler products that converge for smaller $\sigma$. Recall for the Riemann zeta-function we have, for $\sigma>1$, 
\begin{equation} \label{zeta} \zeta(s) = \prod_p \left( 1- \frac{1}{p^s}\right)^{-1} .\end{equation}
From this relation we see that for real $\ell$ and $\sigma \to 1^+$ that 
\begin{equation} \label{zetafactor} \prod_p \left( 1+ \frac{\ell}{p^s}\right) \sim \zeta(s)^\ell , \end{equation}
which tells us the dominant zeta-factor we need to approximate this Euler product.
In what follows we will use the relation, for $\sigma >1$,
\begin{equation} \label{zetainsert} \left(1-\frac{1}{2^{s}} \right) \zeta(s) \prod_{p>2} \left( 1- \frac{1}{p^s}\right) =1 \end{equation}
which can be easily inserted into equations. We will also use the formula, for $\sigma >1$,
\begin{equation} \label{zetainsert_sqf}
\left(1+\frac{1}{2^{s}} \right) \frac{\zeta(2s)}{\zeta(s)} \prod_{p>2} \left( 1+ \frac{1}{p^s}\right) = 1
\end{equation}
It turns out that while we could use \eqref{zetainsert} exclusively, by using \eqref{zetainsert_sqf} at the right time we can skip steps in arriving at the final formulas. Once one finds the final formula by any method, it is then easy to see how to apply \eqref{zetainsert} and \eqref{zetainsert_sqf} optimally.

Returning now to \eqref{g_1}, we see for $\sigma >0$ that as $p\to \infty$,
\[ g_1(p, s) \sim \frac{1}{p^s}, \]
and therefore we need to insert a factor of $\zeta(s)$ into this equation. Applying \eqref{zetainsert} in \eqref{g_1} we have 
\begin{equation} \label{G(s)g_2} \begin{split} G(s) &= \zeta(s) \prod_{p>2} \left(1- \frac{1}{p^s}\right)\left(1+ g_1(p,s)\right) \\ &
= \zeta(s) \prod_{p >2}\left(1+ g_2(p,s)\right), \end{split} \end{equation}
where
\[ 1+ g_2(p,s) = \left(1- \frac{1}{p^s}\right)\left(1+ \left(\frac{p-1}{p-2}\right) \frac{1}{p^s-1}\right)\]
and a calculation gives
\begin{equation} \label{g_2} g_2(p,s) = \frac{1}{(p-2)p^s} . \end{equation}
Thus the Euler product in \eqref{G(s)g_2} converges and gives the analytic continuation of $G(s)$ for $\sigma >0$.

Next, for $p\to \infty$
\[ g_2(p,s) \sim \frac{ 1}{p^{s+1}}, \]
and therefore we need to insert a factor of $\zeta(s+1)$. Using the reciprocal of \eqref{zetainsert_sqf} with $s$ replaced with $ s+1$ in \eqref{G(s)g_2} we obtain
\begin{equation} \label{G(s)g_3}\begin{split} G(s) &= \left(1+\frac{1}{2^{s+1}} \right)^{-1} \frac{\zeta(s)\zeta(s+1)}{\zeta(2s+2)} \prod_{p>2} \left( 1+ \frac{1}{p^{s+1}}\right)^{-1}\left(1+ g_2(p,s)\right) \\ &
= \left(\frac{2^{s+1}}{2^{s+1}+1} \right) \frac{\zeta(s)\zeta(s+1)}{\zeta(2s+2)} \prod_{p>2}\left(1+ g_3(p,s)\right),\end{split} \end{equation}
where
\[ 1+ g_3(p,s) = \left(1+ \frac{1}{p^{s+1}}\right)^{-1}\left(1+\frac{1}{(p-2)p^s}\right), \]
and a calculation gives 
\begin{equation} \label{g_3} g_3(p,s) = \frac{2}{(p-2)(p^{s+1}+1)} .\end{equation}
Here the Euler product in \eqref{G(s)g_3} converges for $\sigma >-1$, which proves 
\eqref{G(s)lemma} and the lemma.
\end{proof}

\section{The Mellin Transform of \texorpdfstring{$E(x)$}{E(x)}} 
The starting point for proving \eqref{SingCAve} is the formula, for $c>1$, 
\begin{equation} \label{Scontour} S(x) = \frac{1}{2\pi i} \int_{c-i\infty}^{c+i\infty} F(s) \frac{x^{s+1}}{s(s+1)} \, ds
\end{equation}
This expresses $S(x)$ as the inverse Mellin transform of $F(s)$, and one now moves the contour to the left and picks up main terms by the residue theorem. This works for $\sigma>-\frac12$, but then, even assuming the Riemann Hypothesis, convergence issues create obstacles for bounding the error term, see \cite{Vaughan2001}.
However, in obtaining oscillation theorems one can proceed in a more elementary way and avoid these issues. 

We now embark on the first step in the proof of Theorem 1 which makes no use of either \eqref{SingCAve} or \eqref{Scontour}. We first compute the Mellin transform of $S(x)$. Since $S(x) \ll x^2\log\log 3x$ by \eqref{gbound}, we have for $\sigma >1$,
\[ \int_1^\infty\frac{S(x)}{x^{s+2}}\, dx = \sum_{k=1}^\infty\mathfrak{S}(k) \int_k^\infty \frac{x-k}{x^{s+2}} \, dx = \frac{1}{s(s+1) } F(s) .\]
We now write
\begin{equation}\label{Edef} E(x) = S(x) - A x^2 -Bx\log x -C x, \end{equation}
where $A$, $B$, and $C$ are constants which will be determined in the course of the proof. Thus we have, for $\sigma >1$, 
\begin{equation} \label{Emellin} \mathcal{E}(s) := \int_1^\infty\frac{E(x)}{x^{s+2}}\, dx = \frac{F(s)}{s(s+1) } - \frac{A}{s-1} -\frac{B}{s^2}-\frac{ C}{s}.\end{equation}
We see that Lemma 2.1 gives the meromorphic continuation of $F(s)$ to the half-plane $\sigma > -1$. Thus for $\sigma >-1$, 
$\frac{F(s)}{s(s+1) } $ has a simple pole at $s=1$ from $\zeta(s)$, a double pole at $s=0$ from $\zeta(s+1)/s$, and poles of order $m_\rho$ at $s= \rho/2-1$ from $1/\zeta(2s+2)$, where $\rho = \beta +i\gamma$ denotes a complex zero of $\zeta(s)$ with multiplicity $m_\rho$. All of these poles from zeta-zeros are in the strip $-1 < \sigma < -1/2$. 
We now choose $A$, $B$, and $C$ so that $\mathcal{E}(s)$ is analytic for $ \sigma \ge -1/2$ and therefore only has the same poles as $F(s)$ at the points $s=\rho/2-1$. 

Clearly $A$ is the residue of the simple pole at $s=1$ of $\frac{F(s)}{s(s+1)}$, and therefore, since $\lim_{s\to 1} (s-1)\zeta(s)=1$, we have from Lemma 2.1
\begin{equation}\label{A} \begin{split} A &= \mbox{\text Res}(\frac{F(s)}{s(s+1)};s=1)\\ &
= \lim_{s\to1} \frac{(s-1)F(s)}{s(s+1)} \\ &
= \frac12\left(\frac{4C_2\zeta(2) \mathcal{G}(1)}{5\zeta(4)}\right) \\ &
= \frac12, \end{split} \end{equation}
since 
\[ \begin{split} \frac{4\zeta(2) \mathcal{G}(1)}{5\zeta(4)} &= \prod_{p>2} 
\left( 1-\frac{1}{p^2}\right)^{-1} \left( 1-\frac{1}{p^4}\right)\left(1+\frac{2}{(p-2)(p^2+1)} \right)\\ &
=\prod_{p>2} 
\left(\frac{(p-1)^2}{p(p-2)}\right) \\ &
= 1/C_2 . \end{split} \]

Next, for the double pole at $s=0$ of $\zeta(s+1)/s$, we let 
\[ U(s) := \frac{s}{\zeta(s+1)}\left( \frac{F(s)}{s(s+1)} \right)
= \frac{4C_2 \zeta(s)\mathcal{G}(s)}{(2^{s+1}+1)\zeta(2s+2)(s+1)}.\]
Since $U(s)$ is analytic at $s=0$ and $\zeta(s) = \frac{1}{s-1} + \gamma + O(|s-1|)$ for $s$ near $1$ (see \cite[Eq. 2.1.16]{Titchmarsh}) we obtain the Laurent expansion around $s=0$
\[\begin{split} \frac{F(s)}{s(s+1)} &= \frac{\zeta(s+1)}{s} U(s)\\& = \left(\frac{1}{s^2} + \frac{\gamma}{s} + O(1)\right)\left( U(0) + sU'(0) + O(|s|^2)\right)\\&
= \frac{U(0)}{s^2} + \frac{\gamma U(0) +U'(0)}{s} +O(1), \end{split} \]
and therefore by \eqref{Emellin} we choose 
\begin{equation} \label{BCchoice} B= U(0), \qquad \text{and} \quad C = \left(\frac{U'(0)}{U(0)} +\gamma\right) U(0). \end{equation}
Here $\gamma \simeq 0.577216$ is the Euler-Mascheroni constant. We will make use below of the special values 
$$
\frac{\zeta'(0)}{\zeta(0)} = \log{2\pi}, \qquad \text{and} \qquad \zeta(0) = - \frac12, $$
see \cite[Sec. 2.4]{Titchmarsh}.

For $B$, we have
\begin{equation} \label{B}
B=U(0) = \frac{4C_2 \zeta(0)\mathcal{G}(0)}{3\zeta(2)}
= -\frac12, 
\end{equation}
since 
\[ \begin{split} \frac{4 \mathcal{G}(0)}{3\zeta(2)} &= \prod_{p>2} 
\left( 1-\frac{1}{p^2}\right) \left(1+\frac{2}{(p-2)(p+1)} \right)\\ &
=\prod_{p>2} 
\left(\frac{(p-1)^2}{p(p-2)}\right) \\ &
= 1/C_2 . \end{split} \]
Next, logarithmically differentiating $U(s)$ and evaluating at $s=0$, we obtain
\[\begin{split}
\frac{U'(0)}{U(0)} 
&= 
 \frac{\zeta'}{\zeta}(0) + \frac{\mathcal{G}'}{\mathcal{G}}(0) -\frac23\log{2}-2\frac{\zeta'}{\zeta}(2) - 1
 \\ & = \log 2\pi - 1, \end{split} \]
since
\[ \begin{split} \frac{\mathcal{G}'}{\mathcal{G}}(0) -\frac23\log{2}-2\frac{\zeta'}{\zeta}(2)
&= \sum_{p>2} \frac{-2p\log{p}}{(p-2)(p+1)^2\left(1+\frac{2}{(p-2)(p+1)} \right)} + \sum_{p>2} \frac{2 \log{p}}{p^2-1} \\& = 0. \end{split} \]
Therefore we conclude by \eqref{BCchoice} and \eqref{B} that
\begin{equation} C = \frac12\left( 1 -\gamma -\log{2\pi}\right) .\end{equation}

We now see that $E(x)$ defined by \eqref{Edef} together with the requirement that $\mathcal{E}(s)$ is analytic in $\sigma \ge -\frac12$ agrees with the definition we gave of $E(x)$ in \eqref{E(x)}.

\section{Landau's Oscillation Theorem and the Completion of the Proof of Theorem 1.}

Landau proved in 1905 that a Dirichlet series with abscissa of convergence $\sigma_c$ which has non-negative coefficients must have a singularity on the real axis at $s=\sigma_c$. From this one sees that if a Dirichlet series with real coefficients does not have a singularity on the real axis at $s=\sigma_c$ then the coefficients must infinitely often take on positive and negative values. This contrapositive form of Landau's theorem is called an oscillation theorem. Some standard references for this method are \cite[Chapter 5]{Ingham1932}, \cite[Chapter 11]{BatemanDiamond2004}, and \cite[Section 15.1]{MontgomeryVaughan2007}.

We will make use of the following form of Landau's theorem which is Lemma 15.1 in \cite{MontgomeryVaughan2007}.

\begin{lemma} \label{Lemma4.1} Let $A(x)$ be a bounded Riemann-integrable function on any finite interval $1\le x\le X$, with $A(x)\ge 0$ for all $x> X_0$. Let $\sigma_c$ be the infimum of the $\sigma$ for which $\int_{X_0}^\infty A(x) x^{-\sigma}\, dx < \infty$. Then the function
\[ F(s) = \int_1^\infty \frac{A(x)}{x^s} \, dx \]
is analytic in the half-plane $\sigma > \sigma_c$ but not at the point $s=\sigma_c$. 
\end{lemma}

Using this we will prove Theorem 1 in the following quantitative form. Let $\rho_1$ be the first zero of the Riemann zeta-function that one encounters when moving up the half-line from the real axis; it is known that $\rho_1 = 1/2 + i \gamma_1 $ is a simple zero, and $\gamma_1 = 14.134725\ldots$.
\begin{theorem} \label{thm2} We have 
\begin{equation}\label{Theorem2} \limsup_{x\to \infty} \frac{E(x)}{x^{1/4}} \ge |c_1|, \qquad \text{and} \qquad \liminf_{x\to \infty} \frac{E(x)}{x^{1/4} }\le - |c_1| \end{equation}
where, letting 
\begin{equation} \label{s1} s_1 = \frac{\rho_1}{2}-1 = -\frac34 +i\frac{\gamma_1}{2},\end{equation} 
\begin{equation} \label{c1} c_1 := \left(\frac{4C_2}{2^{s_1+1}+1} \right) \frac{ \zeta(s_1)\zeta(s_1+1)}{2\zeta'(\rho_1)s_1(s_1+1) }\mathcal{G}(s_1).\end{equation}
\end{theorem}
Using Mathematica we found that $|c_1| = 0.085\ldots$. We will follow the proof given in Lemma 5 of \cite{GV1996} for a similar problem, which is also similar to the proof of Theorem 15.3 in \cite {MontgomeryVaughan2007}, but at the end of the proof we revert to Ingham's argument \cite{Ingham1932} in proving his Theorem 33. 

\begin{proof}[Proof of Theorem \ref{thm2}]
We saw in the last section that $\mathcal{E}(s)$ is analytic for $\sigma \ge -1/2$, and its singularities for $-1 < \sigma < -1/2$ are at the points $ \rho/2-1$ and therefore $\mathcal{E}(s)$ is analytic on the real axis for $ \sigma >-1$.
We now take a constant $c$ satisfying
\begin{equation} \label{c_assump}
0<c<|c_1|,
\end{equation}
where $c_1$ is as defined in \eqref{c1}.

Assume that there is a $X_0 \ge 1$ such that
\begin{equation} \label{Assume} E(x) \le cx^{1/4}, \qquad \text{when} \quad x\ge X_0. \end{equation}
Letting 
\begin{equation} \label{H(s)} H(s) = \int_1^\infty \frac{cx^{1/4} - E(x) }{x^{s+2}}\, dx \end{equation}
we denote the abscissa of convergence of $H(s)$ by $\theta$ and have by \eqref{Emellin} that for $\sigma >\max(-1, \theta)$ 
\begin{equation} \label{H(s)formula} H(s) = \frac{c}{s+ \frac34} -\mathcal{E}(s). \end{equation}
By Lemma \ref{Lemma4.1}, $H(s)$ has a singularity at $s=\theta$, and so
$\theta \le -3/4$. Thus \eqref{H(s)} and \eqref{H(s)formula} hold for $\sigma > -3/4$. 
By \eqref{Assume} and \eqref{H(s)}, we have, for $\sigma > -3/4$, 
\begin{equation} \label{FirstStep}\begin{split} |H(\sigma + it)| &\le \int_1^{X_0} \frac{|cx^{1/4} - E(x)| }{x^{\sigma+2}}\, dx + \int_{X_0}^\infty \frac{cx^{1/4} - E(x) }{x^{\sigma+2}}\, dx \\ &
= \int_1^{X_0} \frac{|cx^{1/4} - E(x)|-(cx^{1/4} - E(x)) }{x^{\sigma+2}}\, dx + H(\sigma)
\\ &
\le 2\int_1^{X_0} \frac{|cx^{1/4} - E(x)| }{x^{\sigma+2}}\, dx +H(\sigma) \\ & = H(\sigma) +O(1).
\end{split}
\end{equation}

The idea used to complete the proof is that the zero $\rho_1$ creates a simple pole for $F(s)$ and therefore for $\mathcal{E}(s)$ and by \eqref{H(s)formula} for $H(s)$ on the line $\sigma = -3/4$ at the point $s_1 = -3/4+ i\gamma_1/2$. But by \eqref{FirstStep} the residue $c$ of the pole of $H(s)$ on the real axis at $\sigma=-3/4$ must be at least as large as the absolute value of the residue of the pole at $s_1$. 

To prove this, we first compute the residue of $\mathcal{E}(s)$ at $s_1$; by \eqref{G(s)lemma} and \eqref{Emellin} 
\[\begin{split} \lim_{\sigma \to {-3/4}^+} (\sigma+3/4)\mathcal{E}(\sigma + i\gamma_1/2) &=
 \left(\frac{4C_2}{2^{s_1+1}+1} \right) \frac{ \zeta(s_1)\zeta(s_1+1)}{s_1(s_1+1) }\mathcal{G}(s_1) \lim_{\sigma \to {-3/4}^+} \frac{ \sigma+3/4}{\zeta(2\sigma +2 +i\gamma_1)}\\&
= \left(\frac{4C_2}{2^{s_1+1}+1} \right) \frac{ \zeta(s_1)\zeta(s_1+1)}{2\zeta'(\rho_1)s_1(s_1+1) }\mathcal{G}(s_1) = c_1,\end{split}\] 
where we used L'H{\^o}pital's rule to evaluate the limit.

Next, in \eqref{H(s)formula} taking $t= \gamma_1/2$ and multiplying both sides by $ \sigma + 3/4$, we obtain 
\[ \lim_{\sigma \to {-3/4}^+} (\sigma+3/4)H(\sigma + i\gamma_1/2)=-c_1 ,\]
and therefore by \eqref{FirstStep}
\[ \begin{split} |c_1| = \left|\lim_{\sigma \to {-3/4}^+} (\sigma+3/4)H(\sigma + i\gamma_1/2)\right| &\le \lim_{\sigma \to {-3/4}^+}(\sigma+3/4) \left|H(\sigma + i\gamma_1/2)\right|\\&
\le \lim_{\sigma \to {-3/4}^+}(\sigma+3/4)(H(\sigma) +O(1)) \\&
= c,\end{split} \]
where we used \eqref{H(s)formula} in the last line. However this contradicts \eqref{c_assump},
and hence \eqref{Assume} is false and
\[\limsup_{x\to \infty} \frac{E(x)}{x^{1/4}} \ge c, \]
for any $c< |c_1|$.
This proves the first inequality in \eqref{Theorem2}, and a concomitant argument proves the second inequality.
\end{proof}



\end{document}